\documentclass[12pt,twoside]{article}
  \usepackage{geometry}
\usepackage{amssymb}
\usepackage{amsmath}
\usepackage{amsthm}
\usepackage{mathrsfs}
\usepackage{cite}
\usepackage{color}
\usepackage{changes}
\usepackage[colorlinks,
linkcolor=blue,
anchorcolor=blue,
citecolor=red
]{hyperref}

 \def\rr{\mathbb{R}}
\def\zz{\mathbb{Z}}
\def\nn{\mathbb{N}}
\def\cp{\mathcal{P}}
\def\cf{\mathcal{F}}

\def\rn{\mathbb{R}^n}
\def\zn{\mathbb{Z}^n}

\def\no{\nonumber}

\def\les{\lesssim}

\newtheorem{thm}{Theorem}[section]
\newtheorem{rem}{Remark}[section]
\newtheorem{lem}{Lemma}[section]

\newtheorem{defn}{Definition}[section]

\def\XXint#1#2#3{{
\setbox0=\hbox{$#1{#2#3}{\int}$}
\vcenter{\hbox{$#2#3$}}\kern-.5\wd0}}

\allowdisplaybreaks[4]
\numberwithin{equation}{section}
\begin{document}
\title{Precompact Sets in Matrix Weighted Lebesgue  Spaces with Variable Exponent \footnotetext{$\ast$ The corresponding author J. S. Xu  jingshixu@126.com\\
The work is supported by the National Natural Science Foundation of China (Grant No. 12161022 and 12061030) and the Science and Technology Project of Guangxi (Guike AD23023002)}}
\author{Shengrong Wang\textsuperscript{a}, Pengfei Guo\textsuperscript{a}, Jingshi Xu\textsuperscript{b,c,d*}\\
{\scriptsize  \textsuperscript{a}School of Mathematics and Statistics, Hainan Normal University, Haikou, 571158, China}\\
{\scriptsize  \textsuperscript{b}School of Mathematics and Computing Science, Guilin University of Electronic Technology, Guilin 541004, China} \\
{\scriptsize  \textsuperscript{c} Center for Applied Mathematics of Guangxi (GUET), Guilin 541004, China}\\
{\scriptsize  \textsuperscript{d}Guangxi Colleges and Universities Key Laboratory of Data Analysis and Computation, Guilin 541004, China}}
\date{}

\maketitle

{\bf Abstract.} In this paper, we first give a sufficiently condition for precompactness in the matrix-weighted Lebesgue spaces with variable exponent by translation operator. Then we obtain a criterion for precompactness in the matrix-weighted Lebesgue space with variable exponent by average operator. Next, we give a criterion for precompactness in the matrix-weighted Lebesgue space with variable exponent by approximate identity.
Finally, precompactness in the matrix-weighted Sobolev space with variable exponent is also considered.



\section{Introduction}
Precompact set in function spaces are often crucial in the proof of the existence of nonlinear partial differential equations.
Kolmogorov first proposed a criterion for compactness in Lebesgue spaces in \cite{ak-1}.
Since then, the compactness criterion for subsets in Lebesgue spaces has been studied in various contexts, such as the generalization of the Kolmogorov-Tulajkov-Sudakov criterion to the case of variable exponential spaces, and the generalization of the Kolmogorov compactness criterion to the weighted case; see \cite{ban-1,ban-2,RH-1,gma-1,dfx-1,fhs-1,lyz}.

In this paper we extend Kolmogorov compactness criterion to the case of  matrix weighted  Lebesgue spaces with variable exponent.
The plan of the paper is as follows. In Section \ref{wbl-s-1}, we collect some notations.
In  section  \ref{wbl-s-5}, we give a sufficient  condition for a precompact set in  matrix-weighted Lebesgue space with variable exponent by translation operator.
In Section \ref{wbl-s-2}, we give the compactness of the average operator  on matrix weighted Lebesgue spaces with variable exponent.
In Section \ref{wbl-s-3}, we give the compactness of identity approximation  on matrix weighted Lebesgue spaces with variable exponent.
In Section \ref{wbl-s-4}, we generalize the results of Sections \ref{wbl-s-5}, \ref{wbl-s-2} and \ref{wbl-s-3} to matrix weighted Sobolev space with variable exponent.

\section{Notations and preliminaries}\label{wbl-s-1}

In this section, we first recall some definitions and notations. Let $\nn$ be the collection of all natural numbers and $\nn_0 = \nn \cup \{0\}$. Let $\zz$ be the collection of all integers.
Let $\rn$ be $n$-dimensional Euclidean space, where $n \in \nn$. Put $\rr= \rr^1$, whereas $\mathbb{C}$ is the complex plane.
In the sequel, $C$ denotes positive constants, but it may change from line to line. For any quantities $A$ and $B$, if there exists a constant $C>0$ such that $A\leq CB$, we write $A \lesssim B$. If $A\lesssim B$ and $B \lesssim A$, we write $A \sim B$.

Let $\mathcal{D}$ be the collection of dyadic cubes in $\mathbb{R}^n.$ For $\nu\in\mathbb{Z}$ and $k\in\mathbb{Z}^n,$ denote by $Q_{\nu k}$ the dyadic cube $2^{-\nu}([0,1)^n+k),$ $x_{Q_{\nu k}}:=2^{-\nu}k$ its ``lower left corner'' and $l(Q_{\nu k})$ its side length.
Let $\mathcal{D}_\nu:=\{Q_{\nu k}:=Q:l(Q)=2^{-\nu},k\in\mathbb{Z}^n\}$.

 For a measurable function $p(\cdot)$ on $\rn$ taking values in $\ (0,\infty)$, we denote ${p^- }: =  {\rm ess}\inf_{x \in \rn} p(x),$ ${p^+ }: =  {\rm ess} \sup_{x \in\rn} p(x).$
The set $\cp(\rn)$ consists of all $p(\cdot)$ satisfying $p^->1$ and $p^+<\infty;$
$\cp_{0}(\rn)$ consists of all $p(\cdot)$ satisfying $p^->0$ and $p^+<\infty$.

Let $p(\cdot)$ be a measurable function on $\rn$ taking values in $(0,\infty)$.
And the Lebesgue space with variable exponent $L^{p(\cdot)}(\rn)$ is defined by
\[L^{p(\cdot)}(\rn): = \bigg\{f \text{ is measurable on}\ \rn:\ \int_{\rn} \bigg( \frac{|f(x)|}
{\lambda } \bigg)^{p(x)}\, {\rm d} x < \infty\text{ for some }\lambda>0 \bigg\}.
\]
Then $L^{p(\cdot)}(\rn)$ is  equipped with the (quasi) norm
\[\| f \|_{L^{p(\cdot)}}: = \inf \bigg\{ \lambda  > 0:\int_{\rn} \bigg( \frac{|f(x)|}{\lambda } \bigg)^{p(x)}{\rm d} x \leq 1  \bigg\}.\]
Denote by $L^\infty(\rn)$ the set of all measurable functions $f$ such that
\[ \|f\|_{L^\infty} := \mathop{{\rm ess}\sup}_{y \in \rn} |f(y)| < \infty.\]
The space $L^{p(\cdot)} _{\rm loc}(\rn)$ is defined by
\[L^{p(\cdot)} _{\rm loc}(\rn):=\{f:f\chi_{K} \in L^{p(\cdot)}(\rn) {\text{ for all compact subsets }} K \subset \rn\},\]
where and what follows, $\chi_{S}$ denotes the characteristic function of a measurable set $S\subset \rn.$

\begin{defn}\label{cd-1}
Let $\alpha(\cdot)$ be a real-valued measurable function on $\rn$.

{\rm (i)} The function $\alpha(\cdot)$ is locally $\log$-H\"{o}lder continuous if there exists a constant $C_{\log}(\alpha)$ such that
\[|\alpha (x) - \alpha (y)| \leq \frac{C_{\log}(\alpha)}{\log ( e + 1/|x - y|)}, \ x,y \in \mathbb{R}^n,\ |x - y| < \frac{1}{2}.\]
Denote by $C_{\rm loc}^{\log}(\rn)$ the set of all locally log-H\"{o}lder continuous

{\rm (ii)} The function $\alpha(\cdot)$ is $\log$-H\"{o}lder continuous at the origin if there exists a constant $C_2$ such that
\[|\alpha (x) - \alpha (0)| \leq \frac{C_2} {\log ( e + 1/| x | )}, \ \forall x \in \rn.\]
Denote by $C^{\log}_0(\rn)$ the set of all $\log$-H\"{o}lder continuous functions at the origin.

{\rm (iii)} The function $\alpha(\cdot)$ is $\log$-H\"{o}lder continuous at infinity if there exists $\alpha_{\infty}\in \mathbb{R}$ and a constant $C_3$ such that
\[|\alpha (x) - \alpha _\infty|\leq \frac{C_3} {\log ( e + | x | )}, \ \forall x \in \rn.\]
Denote by $C^{\log}_\infty(\rn)$ the set of all $\log$-H\"{o}lder continuous functions at infinity.

{\rm (iv)} The function $\alpha(\cdot)$ is global $\log$-H\"{o}lder continuous if $\alpha(\cdot)$ are both locally $\log$-H\"{o}lder continuous and $\log$-H\"{o}lder continuous at infinity.
Denote by $C^{\log}(\rn)$ the set of all global $\log$-H\"{o}lder continuous functions.
\end{defn}

For weights, an important class is the following Muckenhoupt $A_p$ class with constant exponent $p \in [1,\infty]$ which was firstly proposed by Muckenhoupt in \cite{bt22}.

\begin{defn}\label{vd3}
Fix $p \in (1,\infty)$. A positive measurable function $w$ is said to be in the Muckenhoupt class $A_p(\rn)$, if there exists a positive constant $C$ for all balls $B$ in $\rn$ such that
\[ [W]_{A_{p}}:=\sup_{\text{all ball} \ B \subset \rn} \bigg(\frac{1}{|B|}\int_B w(x) {\rm d}x \bigg) \bigg(\frac{1}{|B|}\int_B w(x)^{1-p^{\prime}}{\rm d}x \bigg)^{p-1}<\infty.\]
We say $w \in A_1(\rn)$, if $Mw(x) \leq Cw(x)$ for a.e. $x\in \rn$. We denote $A_\infty :=\cup_{p\geq1} A_p(\rn)$.
\end{defn}

\begin{defn}\label{dg1}
Let $p(\cdot) \in \mathcal{P}(\rn)$. A nonnegative measurable function $w$ is said to be in $A_{p(\cdot)}$,
if there exists a positive constant $C$ for all ball $B$ in $\rn$ such that
\[ \|w\|_{A_{p(\cdot)}}:=\frac{1}{|B|}\|w\chi_B\|_{L^{p(\cdot)}} \|w^{-1}\chi_B\|_{L^{p^{\prime}(\cdot)}} \leq C.\]
\end{defn}

\begin{rem}\label{wbl-R1}
In \cite[Remark 2.3]{cw-1}, Cruz-Uribe and Wang gave that if $w \in A_{p(\cdot)}$, then $w \in L^{p(\cdot)}_{\rm loc}$ and $w^{-1} \in L^{p'(\cdot)}_{\rm loc}$, and if $w \in A_{p(\cdot)}$, then $w^{-1} \in A_{p'(\cdot)}$.
\end{rem}

Now we recall some basic matrix concepts.
For any $m \in \nn$, $M_{m}(\rr)$ is denoted as the set of all $m \times m$ real-valued matrices.
For any $A \in M_m(\rr)$, let
\[|A|_{\rm op}:= \sup_{|\vec{z}|=1} |A\vec{z}|,\]
where $\vec{z}:= (z_1,...,z_m)^{\rm T} \in \rr^m$ and $|\vec{z}| =: (\sum_{i=1}^m |z_i|^2)^{1/2}$, T denotes the transpose of the row vector.

Diagonal matrix $A$ can be denoted as $A={\rm diag}(\lambda_1,...,\lambda_m)$, where $\{\lambda_i\}_{i=1}^m \subset \rr$ .
If $\lambda_1=\cdots =\lambda_m=1$ in the diagonal matrix above, it is called the identity matrix and is denoted by $I_m$.
If there is a matrix $A^{-1} \in M_m(\rr)$ such that $A^{-1}A = I_m$, then matrix $A$ is said to be invertible.

\begin{defn}
A matrix $A  \in M_m(\rr)$ is called positive definite, and if for any $\vec{z}\in \rr^m \setminus \{\vec{0}\}$, $(A\vec{z},\vec{z})>0$. $A  \in M_m(\rr)$ is nonnegative positive definite, and if for any $\vec{z}\in \rr^m \setminus \{\vec{0}\}$, $(A\vec{z},\vec{z}) \geq 0$.
\end{defn}

\begin{defn}
A matrix-valued function $W:\rn \rightarrow M_m(\rr)$ is called a matrix weight if $W$ satisfies that\\
{\rm (i)} for any $x \in \rn, W(x)$ is nonnegative definite;\\
{\rm (ii)} for almost every $x \in \rn$, $W(x)$ is invertible;\\
{\rm (iii)} the entries of $W$ are all locally integrable.
\end{defn}

\begin{defn}\label{wdf-1}
Let $p(\cdot) \in \mathcal{P}(\rn)$. A nonnegative measurable function $W$ is said to be in $\mathcal{A}_{p(\cdot)}$,
if there exists a positive constant $C$ such that
\[ [W]_{\mathcal{A}_{p(\cdot)}}:=\sup_{Q}|Q|^{-1} \big\| \| |W(x) W^{-1}(\cdot)|_{\rm op} \chi_Q(\cdot) \|_{L^{p'(\cdot)}} \chi_Q(x)  \big\|_{L^{p(\cdot)}_x} \leq C,\]
where the supremum is taken over all cubes $Q$ in $\rn$.
\end{defn}

\begin{defn}
Let  $p(\cdot) \in \mathcal{P}(\rn)$. Then $L^{p(\cdot)}(\rn)$ is the set of all measurable vector-valued functions $\vec{f}:=(f_1,f_2, \ldots,f_m)^{\rm T}: \rn\rightarrow \rr^m$ satisfying
\[ \|\vec{f}\|_{L^{p(\cdot)}} := \| |\vec{f}|\|_{L^{p(\cdot)}} < \infty. \]
\end{defn}

\begin{defn}
Let $W$ be a matrix weight from $\rn$ to $M_m(\rr)$ and $p(\cdot) \in \mathcal{P}(\rn)$. Then $L^{p(\cdot)}(W)$ is the set of all measurable vector-valued functions $\vec{f}:=(f_1,f_2, \ldots,f_m)^{\rm T}: \rn\rightarrow \rr^m$ satisfying
\[ \|\vec{f}\|_{L^{p(\cdot)}(W)} := \| |W\vec{f}|\|_{L^{p(\cdot)}} < \infty. \]
\end{defn}

\begin{lem}[see {\cite[Proposition 3.5]{cp-1}}]\label{wbl-L8}
Let $p(\cdot) \in \cp(\rn)$ and $W \in A_{p(\cdot)}$. Then $C^\infty_c(\rn)$ is dense in $L^{p(\cdot)}(W)$.
\end{lem}

\begin{lem}[see {\cite[Proposition 4.8]{cp-1}}]\label{wbl-L9}
Let $W$ be a matrix weight from $\rn$ to $M_m(\rr)$ and $v(x)=|W|_{\rm op}$. If $W \in \mathcal{A}_{p(\cdot)}$, then $v \in A_{p(\cdot)}$.
\end{lem}

We recall the well known Acoli-Arzela theorem, for example, see \cite{af-1} and \cite{db-1}.
What we need is a vector-valued version of the Acoli-Arzela theorem.
\begin{lem}[Acoli-Arzela theorem]\label{wbl-L1}
Let $(\Omega, \rho)$ be a compact metric space and $(E,\|\cdot\|)$ a Banach space.
Denote by $C(\Omega, E)$ the space of all continuous functions from $\Omega$ to $E$ with the supremum norm.
A subset $\cf$ of $C(\Omega, E)$ is precompact in $C(\Omega, E)$ if and only if the following two conditions hold.

{\rm (i)} For every $x \in \Omega$, the set $\{f(x) : f \in \cf\}$ is a precompact set in $E$.

{\rm (ii)} The subset $\cf$ is uniformly equicontinuous: for every $\epsilon > 0$, there exists a $\delta > 0$ such that $|f(x) -f(y)| <\epsilon$ for $f \in \cf$, $x$, $y\in E$ with $\rho(x, y) < \delta$.
\end{lem}

\begin{lem}[see {\cite[Proposition 3.2]{cp-1}}]\label{wbl-L2}
Let $p(\cdot) \in \cp(\rn)$. Then for any $f \in L^{p(\cdot)}(\rn)$, there exists $g \in
L^{p'(\cdot)}(\rn)$ with $\|g\|_{L^{p'(\cdot)}} \leq d$ such that
\[\|\vec{f}\|_{L^{p(\cdot)}}\leq 4 \int_{\rn} \vec{f}(x) \cdot \vec{g}(x)dx.\]
\end{lem}

\section{Translation operators}\label{wbl-s-5}

Let $\tau_y$ denote the translation operator $\tau_y \vec{f}(x): =\vec{f}(x-y)$.

\begin{thm}\label{wbl-T2}
Let $p(\cdot) \in \cp(\rn)$ and $W \in \mathcal{A}_{p(\cdot)}$. Suppose that $\cf$ is a subset in $L^{p(\cdot)}(W)$. Then $\cf$  is precompact if

{\rm (i)} $\cf$ is bounded, i.e., $\sup_{\vec{f} \in \cf} \|\vec{f}\|_{L^{p(\cdot)}(W)} < \infty$.

{\rm (ii)} $\cf$ uniformly vanishes at infinity, that is,
\[\lim_{R \rightarrow \infty} \sup_{\vec{f} \in \cf} \|\vec{f}\chi_{B^c(0,R)}\|_{L^{p(\cdot)}(W)}=0.\]

{\rm (iii)} $\cf$ is equicontinuous, that is,
\[\lim_{r \rightarrow0^+} \sup_{\vec{f} \in \cf} \sup_{y \in B(0,r)} \|\tau_y \vec{f}-\vec{f}\|_{L^{p(\cdot)}(W)}=0\]

\end{thm}
\begin{proof}
Assuming $\cf$ satisfies (i) - (iii).
We want to prove that $\cf$ is totally bounded.
We only need to find a finite $\epsilon$-set of $\cf$ for each fixed $\epsilon>0$  small enough.
Denote by $R_i:=[-2^i,2^i]^n$ for $i \in \zz$.
Then, by  (ii), there exists a sufficiently large positive integer $\beta$ such that
\[ \sup_{\vec{f} \in \cf} \|\vec{f} \chi_{R_\beta}\|_{L^{p(\cdot)}(W)}<\epsilon .\]
By  (iii), there exists an integer $t$ such that
\begin{equation}\label{wbl-16}
\sup_{\vec{f} \in \cf} \sup_{y \in R_t} \|\tau_y \vec{f} -\vec{f}\|_{L^{p(\cdot)}(W)}<\epsilon .
\end{equation}
For some constant $k \in \zn$, there  exists a sequence $\{Q_j\}^N_{j=1}$ of disjoint cubes in $\mathcal{D}_{-t}$ such that $R_\beta=\cup^N_{j=1}Q_j$, where $N=2^{(\beta+1-t)^n}$ and $t \in \zz$ .

For any $\vec{f} \in \cf$ and $x \in \rn$, define
\begin{equation*}
\Phi(\vec{f})(x):=
\begin{cases}
\vec{f}_{Q_j}:= \frac{1}{|Q_j|} \int_{Q_j} \vec{f}(y) dy,  &x \in Q_j, \, j=1,2,\ldots,N,\\
\vec{0},& \text{otherwise.}
\end{cases}
\end{equation*}
If $x \in Q_j$, then $x-Q_j:=\{x-y: y \in Q_j\} \subset R_t$. Thus, we have
\[ \int_{Q_j} |\vec{f}(y)| dy = \int_{x-Q_j} |\vec{f}(x-y)| dy \leq \int_{R_t} |\vec{f}(x-y)| dy.\]
Therefore, we obtain
\begin{align*}
|W(x) (f(x) -f_{Q_j}) \chi_{Q_j}(x)|
& = \bigg| \frac{1}{|Q_j|} \int_{Q_j} W(x) (f(x)-f(z)) dz \chi_{Q_j}(x) \bigg| \\
&\leq \frac{1}{|Q_j|} \int_{Q_j} \big|W(x) (f(x)-f(z)) \big| dz \chi_{Q_j}(x) \\
& \leq \frac{1}{|Q_j|} \int_{R_t}  \big|W(x) (f(x)-f(x-y)) \big| dy \chi_{Q_j}(x).
\end{align*}
It follows that
\begin{align*}
\bigg| \sum_{Q_j \subset R_\beta} (W(x)(f(x) -f_{Q_j})) \chi_{Q_j}(x) \bigg|
&\leq  \sum_{Q_j \subset R_\beta} |W(x) (f(x) -f_{Q_j}) \chi_{Q_j}(x)| \\
&\leq \sum_{Q_j \subset R_\beta} \frac{1}{|Q_j|} \int_{R_t}  \big|W(x) (f(x)-f(x-y)) \big| dy \chi_{Q_j}(x) \\
&=  2^{-nt} \int_{R_t} \big|W(x) (f(x)-f(x-y)) \big| dy \chi_{R_m}(x).
\end{align*}
Therefore, we have
\begin{align*}
\|W(f\chi_{R_\beta}-\Phi(f))\|_{L^{p(\cdot)}}
&= \bigg\| \sum_{Q_j \subset R_\beta}W(f-f_{Q_j}) \chi_{Q_j} \bigg\|_{L^{p(\cdot)}} \\
& \leq 2^{-nt} \bigg\| \int_{R_t} \big|W(x) (f(x)-f(x-y)) \big| dy \chi_{R_\beta}(x) \bigg\|_{L^{p(\cdot)}} \\
& \leq 2^{-nt} \int_{R_t} \|W(\tau_y \vec{f} -\vec{f})\|_{L^{p(\cdot)}} dy\\
& \leq 2^n \sup_{y \in R_t} \|\tau_y \vec{f} -\vec{f}\|_{L^{p(\cdot)}(W)} \\
& <2^n \epsilon,
\end{align*}
where $R_\beta=\cup^N_{j=1}Q_j$, the second inequality and the last inequality are obtained by  the Minkowski-type inequality and (\ref{wbl-16}), respectively. Thus,
\begin{align*}
\|\vec{f}-\Phi(\vec{f})\|_{L^{p(\cdot)}(W)}
&\leq \|(\vec{f}-\Phi(\vec{f})) \chi_{R_\beta}\|_{L^{p(\cdot)}(W)} + \|(\vec{f}-\Phi(\vec{f})) \chi_{R^c_\beta}\|_{L^{p(\cdot)}(W)} \\
&\leq \|(\vec{f}-\Phi(\vec{f})) \chi_{R_\beta}\|_{L^{p(\cdot)}(W)} + \|\vec{f} \chi_{R^c_\beta}\|_{L^{p(\cdot)}(W)} \\
&\les \epsilon.
\end{align*}
It suffices to show that $\Phi(\cf)$ is totally bounded in  $L^{p(\cdot)}(W)$.
By  the above estimates and (i), we have
\[ \sup_{\vec{f} \in \cf} \|\Phi(\vec{f})\|_{L^{p(\cdot)}(W)} \leq \sup_{\vec{f} \in \cf} \|\vec{f}-\Phi(\vec{f})\|_{L^{p(\cdot)}(W)} + \sup_{\vec{f} \in \cf} \|\vec{f}\|_{L^{p(\cdot)}(W)} < \infty.\]
Therefore, for any $\vec{f} \in \cf$,
\[ |W(x) \Phi(\vec{f})(x)| \leq \infty \quad \text{a.e.} \quad x \in \rn.  \]
$W(x)$ is a positive definite matrix, we obtain
\[ |\Phi(\vec{f})(x)| \leq \infty \quad \text{a.e.} \quad x \in \rn,  \]
which implies $|f_{Q_j}|< \infty$, $j=1,2,\ldots,N$. Furthermore,  the entries of $W$ are all locally integrable, we know that $\Phi$ is a mapping from $\cf$ to $\mathcal{B}$, a finite dimensional Banach subspace $L^{p(\cdot)}(W)$.
Note that $\Phi(f)\subset \mathcal{B}$ is bounded and therefore totally bounded.
The proof of Theorem \ref{wbl-T2} is complete.
\end{proof}

\section{Average operators}\label{wbl-s-2}
For any $r \in (0,\infty)$, let $M_r$ denote the average operator
\[M_r \vec{f}(x)=\frac{1}{|B(x,r)|} \int_{B(x,r)} \vec{f}(y) dy, \quad \forall \vec{f} \in L^{p(\cdot)}(W).\]
The boundedness of the average operator on $L^{p(\cdot)}(W)$ has been given in {\cite[Theorems 4.1 and 4.2]{cp-1}} for cubes.
It is easy to see that Theorems 4.1 and 4.2 in \cite{cp-1} hold also for balls.
Thus from Theorems 4.1  in \cite{cp-1}, we have following lemma.
\begin{lem}\label{wbl-L5}
Let $p(\cdot) \in \cp(\rn)$. If $W \in \mathcal{A}_{p(\cdot)}$, then  there exists  a constant $C>0$ such that
\[\|M_r\vec{f}\|_{L^{p(\cdot)}(W)} \leq C[W]_{\mathcal{A}_{p(\cdot)}}  \|\vec{f}\|_{L^{p(\cdot)}(W)}.\]
\end{lem}

\begin{lem}\label{wbl-L6}
Let $p(\cdot) \in \cp(\rn)$. If $W \in \mathcal{A}_{p(\cdot)}$, then  for any $\vec{f} \in L^{p(\cdot)}(W)$
\begin{equation}\label{wbl-2}
 \lim_{r\rightarrow 0} |M_r \vec{f}(x) -\vec{f}(x)|=0 \quad a.e. \ x\in \rn  .
\end{equation}
Moreover, we have that for any $f \in L^{p(\cdot)}(W)$
\[ \lim_{r\rightarrow 0^+} \|M_r\vec{f}-\vec{f}\|_{L^{p(\cdot)}(W)}=0.\]
\end{lem}
\begin{proof}
Let $\vec{f}=(f_1,f_2,\ldots,f_m)^{\rm T} \in L^{p(\cdot)}(W)$. Since $W \in \mathcal{A}_{p(\cdot)}$
\[ |\vec{f}(x)|^{p(x)} = |W(x)^{-1}W(x)\vec{f}(x)|^{p(x)} \leq |W^{-1}(x)|_{\rm op}^{p(x)} |W(x)\vec{f}(x)|^{p(x)}.\]
So then
\[ |\vec{f}(x)|^{p(x)} |W^{-1}(x)|_{\rm op}^{-p(x)} \leq  |W(x)\vec{f}(x)|^{p(x)} .\]
By Lemma \ref{wbl-L9}, $|W^{-1}|_{\rm op} \in A_{p'(\cdot)}$.
Thus $|\vec{f}| \cdot |W^{-1}|_{\rm op} \in L^{p(\cdot)}(\rn)$.
Therefore for each cube $Q$, by H\"{o}lder's inequality we have
\[\|\vec{f}\|_{L(Q)} \leq C \| |\vec{f}| \cdot |W^{-1}|^{-1}_{\rm op} \|_{L^{p(\cdot)}(Q)} \||W^{-1}|_{\rm op}\|_{L^{p'(\cdot)}(Q)}.\]
This means that $|\vec{f}| \in L^1_{\rm loc}(\rn)$. Hence $f_i \in L^1_{\rm loc}(\rn)$ for  every $1\leq i \leq m$.
Now, by the  Lebesgue differentiation theorem, for every $1\leq i \leq m$, we have
\begin{equation}\label{wbl-3}
\lim_{r\rightarrow 0} |M_r f_i(x) -f_i(x)|=0 \quad a.e. \ x\in \rn.
\end{equation}
Furthermore,
\begin{equation}\label{wbl-4}
 |M_r \vec{f}(x) -\vec{f}(x)| \leq  m^{1/2} \max_{1\leq i \leq m} |M_r f_i(x) -f_i(x)| .
\end{equation}
Therefore, from (\ref{wbl-3}) and (\ref{wbl-4}), we obtain (\ref{wbl-2}).

Let $\epsilon>0$. Then by Lemma \ref{wbl-L8}, we can find a function $\vec{g} \in C^\infty_c(\rn)$ such that $\|\vec{f}-\vec{g}\|_{L^{p(\cdot)}(W)} <\epsilon$.
Since $W \in \mathcal{A}_{p(\cdot)}$, then by Lemma \ref{wbl-L5} and (\ref{wbl-2}), for sufficiently small $r$
\[ \|M_r \vec{g} -\vec{g}\|_{L^{p(\cdot)}(W)}  \leq C[W]_{\mathcal{A}_{p(\cdot)}} \|  M_r \vec{g} -\vec{g}\|_{L^{p(\cdot)}} \leq \epsilon.\]
Again by Lemma \ref{wbl-L5}, we have for sufficiently small $r$
\begin{align*}
& \|M_r \vec{f} -\vec{f}\|_{L^{p(\cdot)}(W)}\\
& \quad \leq \|M_r\vec{f} -M_r\vec{g}\|_{L^{p(\cdot)}(W)} + \|M_r \vec{g} -\vec{g}\|_{L^{p(\cdot)}(W)} + \|\vec{g} -\vec{f}\|_{L^{p(\cdot)}(W)}\\
&\quad \leq  \|\vec{g} -\vec{f}\|_{L^{p(\cdot)}(W)} + \|M_r \vec{g} -\vec{g}\|_{L^{p(\cdot)}(W)} \\
& \quad \leq 2\epsilon.
\end{align*}
This completes the proof.
\end{proof}

\begin{thm}\label{t41}
Let $p(\cdot)\in \cp(\rn)$. Suppose that $\cf$ is a subset in $L^{p(\cdot)}(W)$.  Then $\cf$ is a precompact  if and only if

{\rm (i)} $\cf$ is bounded, i. e., $\sup_{\vec{f} \in \cf} \|\vec{f}\|_{L^{p(\cdot)}(W)} < \infty$.

{\rm (ii)} $\cf$ is equicontinuous, that is,
\[\lim_{r \rightarrow0^+} \sup_{\vec{f} \in \cf} \|M_r \vec{f}-\vec{f}\|_{L^{p(\cdot)}(W)}=0.\]

{\rm (iii)} $\cf$ uniformly vanishes at infinity, that is,
\[\lim_{R \rightarrow \infty} \sup_{\vec{f} \in \cf} \|\vec{f}\chi_{B^c(0,R)}\|_{L^{p(\cdot)}(W)}=0.\]

\end{thm}

\begin{proof}
Assuming $\cf$ satisfies (i) - (iii).
By condition (ii), we only need to prove that $M_r \cf$ is precompact for sufficiently  small $r$.
Therefore, we first prove that $M_r \cf$ is a precompact set in $C(\bar{B}(0,R), \rr^m)$ for every $R>0$.
By H\"{o}lder's inequality, we have
\begin{align}\label{wbl-9}
\int_{B(0,2R)} |\vec{f}(y)| dy
&=   \int_{B(0,2R)} |W(y) W(y)^{-1}\vec{f}(y)| dy \no\\
&\leq \|\vec{f}\chi_{B(0,2R)}\|_{L^{p(\cdot)}(W)} \||W^{-1}|_{\rm op} \chi_{B(0,2R)}\|_{L^{p'(\cdot)}}.
\end{align}
Thus, for any $\vec{f} \in \cf$ and any fixed $x \in \bar{B}(0,R)$, we have
\begin{align}\label{wbl-1}
|M_r\vec{f}(x)|
&\les \frac{1}{|B(x,r)|} \int_{B(x,r)} |\vec{f}(y)|dy \nonumber\\
&\les \frac{1}{|B(x,r)|} \int_{B(0,2R)} |\vec{f}(y)|dy \nonumber\\
&\les \frac{1}{|B(0,r)|} \|\vec{f}\chi_{B(0,2R)}\|_{L^{p(\cdot)}(W)} \||W^{-1}|_{\rm op} \chi_{B(0,2R)}\|_{L^{p'(\cdot)}},
\end{align}
where we used the fact that $B(x,r) \subset B(0,2R)$ when $x \in \bar{B}(0,R)$, and $r \in (0,R)$.
Since $W \in \mathcal{A}_{p(\cdot)}$,  $W^{-1} \in \mathcal{A}_{p'(\cdot)}$.
Therefore, by Lemma \ref{wbl-L9}, $|W^{-1}|_{\rm op} \in A_{p'(\cdot)}$. Thus, $\||W^{-1}|_{\rm op} \chi_{B(0,2R)}\|_{L^{p'(\cdot)}}< \infty$.
Thus, $M_r\cf$ is uniformly bounded. Next, we show $M_r \cf$ is equicontinuous. For any fixed $x$, $y \in \bar{B}(0,R)$, by (\ref{wbl-9}) we have
\begin{align}\label{wbl-8}
|M_r\vec{f}(y)-M_r\vec{f}(x)|
 &= \bigg| \frac{1}{|B(x,r)|} \int_{B(y,r)} \vec{f}(z) dz - \frac{1}{|B(x,r)|} \int_{B(x,r)} \vec{f}(z) dz \bigg| \nonumber\\
&   \les \frac{1}{|B(x,r)|} \int_{B(x,r) \Delta B(y,r)} |\vec{f}(z)| dz.
\end{align}
For any $x,y \in \bar{B}(0,R)$ and $\vec{f} \in \cf$, by (\ref{wbl-9}) and (\ref{wbl-1}), we obtain
\begin{align}\label{wbl-10}
&\int_{B(x,r) \Delta B(y,r)} |\vec{f}(z)| dz \no\\
& \quad \les  \int_{B(x,r) \Delta B(y,r)} |\vec{f}(z)-M_r\vec{f}(z)| dz + \int_{B(x,r) \Delta B(y,r)} |M_r\vec{f}(z)| dz \no\\
& \quad \les \||W^{-1}|_{\rm op} \chi_{B(x,r) \Delta B(y,r)}\|_{L^{p'(\cdot)}} \|M_r\vec{f}-\vec{f}\|_{L^{p(\cdot)}(W)} \no\\
&\quad \quad +  \||W^{-1}|_{\rm op} \chi_{B(x,r) \Delta B(y,r)}\|_{L^{p'(\cdot)}} \|\vec{f}\|_{L^{p(\cdot)}(W)},
\end{align}
Because  $|B(x,r) \Delta B(y,r)|$ converges to $0$ as $x$ converges to $y$, then  $\||W^{-1}|_{\rm op} \chi_{B(x,r) \Delta B(y,r)}\|_{L^{p'(\cdot)}}$ $\rightarrow 0$ as $x$ converges to $y$. Therefore by (\ref{wbl-8}) and (\ref{wbl-10}), we obtain $M_r\cf$ is equicontinuous on $\bar{B}(0,R)$.
According to condition (i), for each $x \in \bar{B}(0,R)$, $\{M_r\vec{f} (x): \vec{f} \in \cf\}$ is a precompact set in $\rn$. Hence by Lemma \ref{wbl-L1}, $M_r\cf$ is precompact in $C(\bar{B}(0,R), \rr^m)$.

Finally, we verify that $M_r\cf$ has finite $\epsilon$-net for each $\epsilon \in  (0,1)$.
By condition (iii), there exists  $R > 0$ such that
\begin{equation}\label{wbl-6}
 \sup_{\vec{f} \in \cf} \|\vec{f} \chi_{B^c(0,R)}\|_{L^{p(\cdot)}(W)}<\epsilon/8.
\end{equation}
By condition (ii), we choose $r$ enough small  such that
\begin{equation}\label{wbl-7}
 \sup_{\vec{f} \in \cf} \|M_r\vec{f}-\vec{f}\|_{L^{p(\cdot)}(W)}<\epsilon / 8.
\end{equation}
By the fact that $M_r\cf$ is precompact in $C(\bar{B}(0,R), \rr^m)$, which was proved above, there exists $\{\vec{f}_1, \vec{f}_2, \ldots , \vec{f}_k\} \subset \cf$
such that $\{M_r\vec{f}_1, M_h\vec{f}_2,\ldots, M_h\vec{f}_k\}$
is a finite $\epsilon/ \gamma$-net of $M_r\cf$
in the norm of $C(\bar{B}(0,R), \rr^m)$, where $\gamma= 2 \| |W|_{\rm op} \chi_{B(0,R)}\|_{L^{p(\cdot)}}$.
i.e. for each $\vec{f} \in \cf$, we choose $\vec{f}_j$, $j\in \{1,2,\ldots,k\}$ such that
 \begin{equation}\label{wbl-11}
 \sup_{y \in B(0,R)} |M_r \vec{f}(y) -M_r\vec{f}_j(y)|<\epsilon/ \gamma .
 \end{equation}
Below we verify that $\{M_r\vec{f}_1,$ $M_r\vec{f}_2,\ldots, M_r\vec{f}_k\}$ is a finite $\epsilon$-net of $M_r \cf$ in the norm of $L^{p(\cdot)}(W)$.
By (\ref{wbl-11}), we obtain
\begin{align}\label{wbl-12}
 \|(M_r\vec{f} - M_r \vec{f}_j)\chi_{B(0,R)}\|_{L^{p(\cdot)}(W)}
& \leq  \| |W|_{\rm op} |(M_r \vec{f} - M_r \vec{f}_j)\chi_{B(0,R)}|\|_{L^{p(\cdot)}} \no\\
&\leq \| |W|_{\rm op} \chi_{B(0,R)}\|_{L^{p(\cdot)}}  \sup_{x \in B(0,R)} |M_r \vec{f}(x) -M_r\vec{f}_j(x)| \no\\
&\leq \epsilon/2.
\end{align}
Then, by  (\ref{wbl-6}), (\ref{wbl-7}) and (\ref{wbl-12})
\begin{align*}
&\|M_r \vec{f} - M_r \vec{f}_j\|_{L^{p(\cdot)}(W)} \\
&\quad \leq \|(M_r\vec{f} - M_r \vec{f}_j)\chi_{B(0,R)}\|_{L^{p(\cdot)}(W)} + \|(M_r \vec{f} - M_r \vec{f}_j)\chi_{B^c(0,R)}\|_{L^{p(\cdot)}(W)}\\
&\quad \leq \epsilon/2 + \|M_r\vec{f}-\vec{f}\|_{L^{p(\cdot)}(W)}+ \|\vec{f}_j-M_r\vec{f}_j\|_{L^{p(\cdot)}(W)}\\
& \quad \quad + \|\vec{f}\chi_{B^c(0,R)}\|_{L^{p(\cdot)}(W)} + \|\vec{f}_j\chi_{B^c(0,R)}\|_{L^{p(\cdot)}(W)} \\
&\quad \leq  \epsilon.
\end{align*}

Conversely, suppose that $\cf$ is a precompact set in $L^{p(\cdot)}(W)$. Then $\cf$ is a bounded set.
Now for any given $\epsilon>0$,  let $\{\vec{f}_1,\vec{f}_2,...,\vec{f}_N\}$ be an $\epsilon$-net of $\cf$.

 For   each $1 \leq k \leq N$, since $\vec{f}_k\in L^{p(\cdot)}(W)$,  by the monotone convergence theorem, there exists $H_k>0$ such that
\[ \|\vec{f}_k\chi_{B^c(0,R)}\|_{L^{p(\cdot)}(W)} < \epsilon/2 \quad \text{for } R>H_k.\]
Set $H=\max\{R_k: 1 \leq k \leq N\}$, then
\[ \|\vec{f}\chi_{B^c(0,R)}\|_{L^{p(\cdot)}(W)} \leq  \|(\vec{f}-\vec{f}_k)\chi_{B^c(0,R)}\|_{L^{p(\cdot)}(W)} + \|\vec{f}_k\chi_{B^c(0,R)}\|_{L^{p(\cdot)}(W)} < \epsilon \]
for $R>H$. Thus, (iii) holds.

 By Lemma \ref{wbl-L6}, there exists $r>0$ such that for $h< r$,
\[ \max_{1 \leq k \leq N} \|M_h\vec{f}_k - \vec{f}_k \|_{L^{p(\cdot)}(W)} < \epsilon. \]
Then by Lemma \ref{wbl-L5},
 \begin{align*}
&\|M_h\vec{f}-\vec{f}\|_{L^{p(\cdot)}(W)} \\
&\quad \leq  \|M_h(\vec{f}-\vec{f}_k)\|_{L^{p(\cdot)}(W)} + \|\vec{f}-\vec{f}_k\|_{L^{p(\cdot)}(W)} + \|M_h\vec{f}_k-\vec{f}_k\|_{L^{p(\cdot)}(W)}\\
&\quad \les \|\vec{f}-\vec{f}_k\|_{L^{p(\cdot)}(W)} + \|M_h\vec{f}_k-\vec{f}_k\|_{L^{p(\cdot)}(W)} \\
&\quad \les \epsilon.
\end{align*}
Thus, (ii) holds. This completes the proof.
\end{proof}

\section{By approximate identity}\label{wbl-s-3}
In this section, we give a sufficient and necessary condition for a precompact set in  matrix-weighted Lebesgue space with variable exponent by approximate identity. To state our results we need the following lemma.

\begin{lem}[see {\cite[Theorem 1.3]{cp-1}}]\label{wbl-L7}
Let $p(\cdot) \in \cp(\rn)\cap C^{\log}(\rn)$ and $W \in \mathcal{A}_{p(\cdot)}$. Let $\phi \in C^\infty_c(B(0,1))$ be
a nonnegative, radially symmetric and decreasing function with $\int_{\rn} \phi(x) =1$. Let $\phi_r(x) =r^{-n} \phi(x/r)$, $r>0$.
Then  for each $\vec{f} \in L^{p(\cdot)}(W)$,
\[\sup_{r>0} \|\phi_r \ast \vec{f}\|_{L^{p(\cdot)}(W)} \leq C [W]_{\mathcal{A}_{p(\cdot)}} \| \vec{f}\|_{L^{p(\cdot)}(W)}.\]
Furthermore, for every $\vec{f} \in L^{p(\cdot)}(W)$,
\[ \lim_{r \rightarrow 0}\|\phi_r \ast \vec{f}-\vec{f}\|_{L^{p(\cdot)}(W)} =0.\]
\end{lem}

\begin{thm}\label{t51}
Let $p(\cdot) \in \cp(\rn)\cap C^{\log}(\rn)$ and $W \in \mathcal{A}_{p(\cdot)}$.
Let $p(\cdot) \in \cp(\rn)\cap C^{\log}(\rn)$ and $W \in \mathcal{A}_{p(\cdot)}$. Let $\phi\in C^\infty_c(B(0,1))$ be a nonnegative, radially symmetric and decreasing function with $\int_{\rn} \phi(x) =1$. Let $\phi_r(x) =r^{-n} \phi(x/r)$, $r>0$.
Suppose that $\cf$ is a  subset in $L^{p(\cdot)}(W)$. Then $\cf$  is precompact if and only if

{\rm (i)} $\cf$ is bounded, i. e., $\sup_{\vec{f} \in \cf} \|\vec{f}\|_{L^{p(\cdot)}(W)} < \infty$.

{\rm (ii)} $\cf$ is equiapproximate, that is,
\[\lim_{r \rightarrow0^+} \sup_{\vec{f} \in \cf} \|\vec{f}\ast \phi_r-\vec{f}\|_{L^{p(\cdot)}(W)}=0.\]

{\rm (iii)} $\cf$ uniformly vanishes at infinity, that is,
\[\lim_{R \rightarrow \infty} \sup_{f \in \cf} \|f\chi_{B^c(0,R)}\|_{L^{p(\cdot)}(W)}=0.\]
\end{thm}

\begin{proof}
Assuming $\cf$ satisfies (i) - (iii).
By condition (ii), we only need to prove that $\phi_r \ast \cf :=\{\phi_r \ast \vec{f}: \vec{f} \in \cf\}$ is a precompact for sufficiently  small $r$.
To do so we first prove that $\phi_r \ast \cf$ is a precompact set in $C(\bar{B}(0,R), \rr^m)$ for every $R>0$.
Since $W \in \mathcal{A}_{p(\cdot)}$, $W^{-1} \in \mathcal{A}_{p'(\cdot)}$. Thus, by Lemma \ref{wbl-L9}, $|W^{-1}|_{\rm op} \in A_{p'(\cdot)}$.
Therefore, by Remark \ref{wbl-R1},  $|W^{-1}|_{\rm op} \in L^{p'(\cdot)}_{\rm loc}(\rn)$.
Then by H\"{o}lder's inequality, we have
\begin{align}\label{wbl-14}
|\vec{f}\ast \phi_r(x)|
& = \int_{B(0,R+r)} \vec{f}(y)\phi_r(x-y) dy  \no \\
&\leq \|\phi_r\|_{L^\infty} \int_{B(0,R+r)} |W^{-1}(y) W(y)\vec{f}(y)| dy \no \\
& \leq \|\phi_r\|_{L^\infty} \||W^{-1}|_{\rm op} \chi_{B(0,R+r)}\|_{L^{p'(\cdot)}} \|\vec{f}\|_{L^{p(\cdot)}(W)}.
\end{align}
Thus, $\phi_r \ast \cf$ is uniformly bounded. Next, we show $\phi_r \ast \cf$ is equicontinuous. For any $x_1$, $x_2 \in \bar{B}(0,R)$, by  (\ref{wbl-14}), we have
\begin{align}\label{wbl-15}
& |(\vec{f}\ast \phi_r(x_1) - \vec{f}\ast \phi_r(x_2))| \no \\
& \quad\leq  \int_{B(0,R+r)} |\vec{f}(y) (\phi_r(x_1-y)-\phi_r(x_2-y))| dy \no \\
& \quad \leq \||W^{-1}|_{\rm op} \chi_{B(0,R+r)}\|_{L^{p'(\cdot)}} \|\vec{f}\|_{L^{p(\cdot)}(W)} \sup_{y\in B(0,r+R)} |\phi_r(x_1-y)-\phi_r(x_2-y)|.
\end{align}
It remains to show that
\begin{equation}\label{wbl-5}
\sup_{y\in B(0,r+R)} |\phi_r(x_1-y)-\phi_r(x_2-y)| <\epsilon/ \||W^{-1}|_{\rm op} \chi_{B(0,R+r)}\|_{L^{p'(\cdot)}}.
\end{equation}
It follows that
\begin{align*}
|\phi_r(x_1-y)-\phi_r(x_2-y)| \leq \frac{1}{r^n} \Big|\phi \big(\frac{x_1-y}{r} \big)-\phi \big(\frac{x_2-y}{r} \big) \Big|.
\end{align*}
For any $\epsilon \in (0,1)$, since $\phi$ is uniformly continuous on the unit ball, there is a $\delta>0$ such that
\[|\phi(z_1)-\phi(z_2)| < r^n\epsilon/ \||W^{-1}|_{\rm op} \chi_{B(0,R+r)}\|_{L^{p'(\cdot)}}<1, \ |z_1-z_2|<\delta.\]
Thus, if $|x_1-x_2|<r\delta$, then we have
\[\Big|\phi \big(\frac{x_1-y}{r} \big)-\phi \big(\frac{x_2-y}{r} \big) \Big| < r^n\epsilon/ \||W^{-1}|_{\rm op} \chi_{B(0,R+r)}\|_{L^{p'(\cdot)}}<1 .\]
Hence,  we obtain (\ref{wbl-5}). By (ii), (\ref{wbl-15}) and (\ref{wbl-5}), we obtain $\phi_r\ast \cf$ is uniformly equicontinuous.

Finally, it remains to show that $\cf$ has a finite $\epsilon$-net for $\epsilon \in (0,1)$.
By  (iii), there exists $R > 0$ such that, for each $\vec{f} \in \cf$,
\[ \sup_{\vec{f} \in \cf} \|\vec{f} \chi_{B^c(0,R)}\|_{L^{p(\cdot)}(W)}<\epsilon/5 .\]
By  (ii), we choose $r$ small enough so that, for each $\vec{f} \in \cf$,
\[ \sup_{\vec{f} \in \cf} \|\vec{f}\ast \phi_r -\vec{f}\|_{L^{p(\cdot)}(W)}<\epsilon / 5.\]
By the fact that $\phi_r\ast \cf$ is precompact in $C(\bar{B}(0,R), \rr^m)$, which was proved above, there exists $\{\vec{f}_1, \vec{f}_2, \ldots , \vec{f}_N\} \subset \cf$
such that $\{\vec{f}_1 \ast \phi_r, \vec{f}_2\ast \phi_r, \ldots , \vec{f}_N\ast \phi_r\}$
is a finite $\epsilon/\beta$-net of $\phi_r\ast \cf$ in the norm of $C(\bar{B}(0,R), \rr^m)$, where $\beta=5\| |W|_{\rm op} \chi_{B(0,R)}\|_{L^{p(\cdot)}}$.
Thus, we have
\begin{align}\label{wbl-13}
& \|(\vec{f}\ast \phi_r -\vec{f}_j\ast \phi_r)\chi_{B(0,R)}\|_{L^{p(\cdot)}(W)} \no\\
& \quad \leq  \| |W|_{\rm op} |(\vec{f}\ast \phi_r -\vec{f}_j\ast \phi_r)\chi_{B(0,R)}|\|_{L^{p(\cdot)}} \no\\
& \quad \leq \| |W|_{\rm op} \chi_{B(0,R)}\|_{L^{p(\cdot)}}  \sup_{y \in B(0,R)} |\vec{f}\ast \phi_r(y) -\vec{f}_j\ast \phi_r(y)| \no\\
&\quad  \leq \epsilon/5.
\end{align}
Then, by (\ref{wbl-13})
\begin{align*}
&\|\vec{f} - \vec{f}_j\|_{L^{p(\cdot)}(W)}\\
& \leq \|(\vec{f} - \vec{f}_j) \chi_{B(0,R)}\|_{L^{p(\cdot)}(W)} +\|(\vec{f} - \vec{f}_j) \chi_{B^c(0,R)}\|_{L^{p(\cdot)}(W)} \\
& \leq \|(\vec{f} - \vec{f} \ast \phi_r)\chi_{B(0,R)}\|_{L^{p(\cdot)}(W)} + \|(\vec{f} \ast \phi_r - \vec{f}_j \ast \phi_r)\chi_{B(0,R)}\|_{L^{p(\cdot)}(W)}\\
& \quad + \|(\vec{f}_j \ast \phi_r - \vec{f}_j)\chi_{B(0,R)}\|_{L^{p(\cdot)}(W)} + \|\vec{f}\chi_{B^c(0,R)}\|_{L^{p(\cdot)}(W)} + \|\vec{f}_j\chi_{B^c(0,R)}\|_{L^{p(\cdot)}(W)}\\
& \leq  \epsilon.
\end{align*}

Conversely, if $\cf$ is a precompact set in $L^{p(\cdot)}(W)$, obviously, (i) is true.
As for (ii), for each $\epsilon>0$, let $\{\vec{f}_k\}_{k=1}^N \subseteq \cf$ be a $\epsilon$-net of $\cf$. By Lemma \ref{wbl-L7}, there exists $r>0$ such that for $t < r$,
\[ \max_{1\leq k\leq N} \|\vec{f}_k \ast\phi_t -\vec{f}_k\|_{L^{p(\cdot)}(W)} < \epsilon .\]
Then, by Lemma \ref{wbl-L7} again, we obtain
\begin{align*}
\|\vec{f} \ast\phi_t -\vec{f}\|_{L^{p(\cdot)}(W)}
& \leq \|\vec{f} \ast\phi_t -\vec{f}_k \ast \phi_t\|_{L^{p(\cdot)}(W)} +  \|\vec{f}_k \ast \phi_t-\vec{f}_k\|_{L^{p(\cdot)}(W)} \\
& \quad + \|\vec{f}_k-\vec{f}\|_{L^{p(\cdot)}(W)}\\
& \les \|\vec{f}_k \ast \phi_t-\vec{f}_k\|_{L^{p(\cdot)}(W)} + \|\vec{f}_k-\vec{f}\|_{L^{p(\cdot)}(W)} \\
&\les \epsilon .
\end{align*}

(iii) is proved in Theorem \ref{t41}.

Thus the proof is finished.\end{proof}

\section{Precompactness in Sobolev space}\label{wbl-s-4}
In this section, we consider the precompactness in matrix weighted Sobolev spaces with variable exponent. To go on, we recall some notions.

\begin{defn}
Suppose that $f \in L^1_{\rm loc}(\rn)$. If there exists a function $g_j \in L^1_{\rm loc}(\rn)$  such that for any $\phi \in C^\infty_0(\rn)$,
\[ \int_{\rn} f(x) \partial_j \phi(x) dx = -\int_{\rn} g_j(x) \phi(x) dx . \]
Then $g_j$ is called the weak derivative of $f$ with respect to $x_j$. Denote $g_j$ by $\partial_jf$.
\end{defn}

If $\vec{f}=(f_1,f_2,\ldots,f_m)^{\rm T} \in L^1_{\rm loc}(\rn)$, then define the weak derivative of $\vec{f}$ with respect to $x_j$ by $\partial_j \vec{f}=(\partial_j f_1,\partial_j f_2,\ldots, \partial_j f_m)^{\rm T}$.
 Denote by $\mathcal{W}^{1,1}_{\rm loc}(\rn)$ the set of function $f \in L^1_{\rm loc}(\rn)$ such that $\partial_j f_i$ exists, $j=1,2,\ldots,n$, $i=1,2,\ldots,m$.
 Let $D \vec{f}(\partial_j f_i)_{i,j}$ be the Jacobian matrix of $\vec{f}$.

\begin{defn}
Let $p(\cdot) \in \cp(\rn)$ and $W \in \mathcal{A}_{p(\cdot)}$.
The  $\mathcal{W}^{1,p(\cdot)}(W)$ is the collection of $\vec{f} \in \mathcal{W}^{1,1}_{\rm loc}(\rn)$ such that
\[ \|\vec{f}\|_{\mathcal{W}^{1,p(\cdot)}(W)} :=  \|\vec{f}\|_{L^{p(\cdot)}(W)} + \sum^n_{j=0}\|\partial \vec{f}_j\|_{L^{p(\cdot)}(W)} < \infty.\]
\end{defn}

\begin{thm}\label{wbl-T3}
Let $p(\cdot) \in \cp(\rn)$ and $W \in \mathcal{A}_{p(\cdot)}$. Suppose that $\cf$ is a  subset in $\mathcal{W}^{1,p(\cdot)}(W)$. If

{\rm (i)} $\cf$ is bounded, i. e., $\sup_{\vec{f} \in \cf} \|\vec{f}\|_{\mathcal{W}^{1,p(\cdot)}(W)} < \infty$.

{\rm (ii)} $\cf$ is equicontinuous, that is,
\[\lim_{r \rightarrow0^+} \sup_{\vec{f} \in \cf} \sup_{y \in B(0,r)} \|\tau_y \vec{f}-\vec{f}\|_{\mathcal{W}^{1,p(\cdot)}(W)}=0 .\]

{\rm (iii)} $\cf$ uniformly vanishes at infinity, that is,
\[\lim_{R \rightarrow \infty} \sup_{\vec{f} \in \cf} \|\vec{f}\chi_{B^c(0,R)}\|_{\mathcal{W}^{1,p(\cdot)}(W)}=0.\]
Then $\cf$  is a precompact set in $\mathcal{W}^{1,p(\cdot)}(W)$.
\end{thm}

\begin{proof}
Observe that $\cf$ is precompact in $\mathcal{W}^{1,p(\cdot)}(W)$ if and only if $\cf$ and $\{\partial_j\vec{f}: \vec{f} \in \cf\}$ are precompact in $L^{p(\cdot)}(W)$ for $j=1,2,\ldots,n$. Thus, we have Theorem \ref{wbl-T3} by Theorem \ref{wbl-T2}.
\end{proof}

Similarly, by Theorems \ref{t41} and \ref{t51} we have the following results.
\begin{thm}
Let $p(\cdot) \in \cp(\rn)$ and $W \in \mathcal{A}_{p(\cdot)}$. Suppose that $\cf$ is a  subset in $\mathcal{W}^{1,p(\cdot)}(W)$. If

{\rm (i)} $\cf$ is bounded, i. e., $\sup_{\vec{f} \in \cf} \|\vec{f}\|_{\mathcal{W}^{1,p(\cdot)}(W)} < \infty$.

{\rm (ii)} $\cf$ is equicontinuous, that is,
\[\lim_{r \rightarrow0^+} \sup_{\vec{f} \in \cf}  \|M_r \vec{f}-\vec{f}\|_{\mathcal{W}^{1,p(\cdot)}(W)}=0 .\]

{\rm (iii)} $\cf$ uniformly vanishes at infinity, that is,
\[\lim_{R \rightarrow \infty} \sup_{\vec{f} \in \cf} \|f\chi_{B^c(0,R)}\|_{\mathcal{W}^{1,p(\cdot)}(W)}=0.\]
Then $\cf$  is a precompact set in $\mathcal{W}^{1,p(\cdot)}(W)$.
\end{thm}

\begin{thm}
Let $p(\cdot) \in \cp(\rn)\cap C^{\log}(\rn)$ and $W \in \mathcal{A}_{p(\cdot)}$.
Let $p(\cdot) \in \cp(\rn)\cap C^{\log}(\rn)$ and $W \in \mathcal{A}_{p(\cdot)}$. Let $\phi\in C^\infty_c(B(0,1))$ be a nonnegative, radially symmetric and decreasing function with $\int_{\rn} \phi(x) =1$. Let $\phi_r(x) =r^{-n} \phi(x/r)$, $r>0$. Suppose that $\cf$ is a subset in $\mathcal{W}^{1,p(\cdot)}(W)$. Then $\cf$  is precompact if and only if

{\rm (i)} $\cf$ is bounded, i. e., $\sup_{\vec{f} \in \cf} \|\vec{f}\|_{\mathcal{W}^{1,p(\cdot)}(W)} < \infty$.

{\rm (ii)} $\cf$ is equicontinuous, that is,
\[\lim_{r \rightarrow 0^+} \sup_{\vec{f} \in \cf}  \|\vec{f}\ast\phi_r-\vec{f}\|_{\mathcal{W}^{1,p(\cdot)}(W)}=0 .\]

{\rm (iii)} $\cf$ uniformly vanishes at infinity, that is,
\[\lim_{R \rightarrow \infty} \sup_{\vec{f} \in \cf} \|\vec{f}\chi_{B^c(0,R)}\|_{\mathcal{W}^{1,p(\cdot)}(W)}=0.\]
\end{thm}

\end{document}